\documentclass[12pt]{amsart}
\usepackage{amssymb,amsmath,amsthm}
\usepackage{url}
\usepackage{mathtools}
\usepackage{tikz}

\usetikzlibrary{arrows,shapes,automata,backgrounds,decorations,petri,positioning,patterns}

\def\drawpolygon#1,#2;{
    \begin{pgfonlayer}{background}
        \draw[gray,line width=20,join=round      ](#1.center)foreach\A in{#2}{--(\A.center)}--cycle;
        \filldraw[line width=19,join=round,white](#1.center)foreach\A in{#2}{--(\A.center)}--cycle;
    \end{pgfonlayer}
}

\usepackage{enumitem}

\theoremstyle{plain}
\theoremstyle{definition}
\newtheorem{theorem}{Theorem}[section]

\newtheorem{lemma}[theorem]{Lemma}

\newtheorem{definition}[theorem]{Definition}

\newtheorem{question}[theorem]{Question}
\newtheorem{example}[theorem]{Example}
\newtheorem{proposition}[theorem]{Proposition}
\newtheorem{corollary}[theorem]{Corollary}

\newtheorem*{repp@ex}{\repp@title (continued)}
\newcommand{\newreppex}[2]{
\newenvironment{repp#1}[1]{
 \def\repp@title{#2 \ref{##1}}
 \begin{repp@ex}}
 {\end{repp@ex}}}
\makeatother
\newreppex{ex}{Example}

\newcommand{\mf}[1]{\mbox{$\mathfrak #1$}}
\newcommand{\poset}{\mbox{$\mathcal{P}$}}
\newcommand{\hatposet}{\mbox{$\overline{\mathcal{P}}$}}
\newcommand{\simple}{\mbox{$\textsf{simp}$}}

\newcommand{\interval}{\mbox{$\mathcal{I}$}}
\newcommand{\ub}[1]{\underbracket[.5pt][1pt] {#1}}
\newcommand{\ob}[1]{\overbracket[.5pt][1pt] {#1}}

\begin{document}

\title{Interval posets of permutations}

\date{}

\author{Bridget Eileen Tenner}
\address{Department of Mathematical Sciences, DePaul University, Chicago, IL 60614}
\email{bridget@math.depaul.edu}
\thanks{Research partially supported by Simons Foundation Collaboration Grant for Mathematicians 277603 and by a DePaul University Faculty Summer Research Grant.\\
\indent Data sharing not applicable to this article as no datasets were generated or analyzed during the current study.}



\begin{abstract}
The interval poset of a permutation catalogues the intervals that appear in its one-line notation, according to set inclusion. We study this poset, describing its structural, characterizing, and enumerative properties.
\end{abstract}

\maketitle

\section{Introduction}\label{sec:intro}

A permutation is a bijection between two totally ordered sets. Typically we focus on $\mf{S}_n$, the set of permutations of $[n] := \{1,\ldots, n\}$. For the purposes of this work, we write permutations as words in one-line notation
$$w = w(1)w(2) w(3) \cdots,$$
and we take an ``interval'' in a permutation to be an interval of values.

\begin{definition}\label{defn:interval}
An \emph{interval} in a permutation $w$ is an interval of values (possibly empty) that appear in consecutive positions of $w$. That is, $[h,h+j]$ is an interval of $w$ if $\{w(t) : t \in [i,i+j]\} = [h,h+j]$, for some $i$.
\end{definition}

Clearly $[1,n]$ is an interval of every $w \in \mf{S}_n$, as is $\{i\}$ for each $i \in [n]$. An interval of $w \in \mf{S}_n$ is \emph{proper} if it has between $2$ and $n-1$ elements.

\begin{example}\label{ex:43187562}
The proper intervals of $43187562$ are $[3,4]$, $[5,6]$, $[5,7]$, $[5,8]$, and $[7,8]$.
\end{example}

A permutation need not have any proper intervals, as is the case for $2413 \in \mf{S}_4$ and $35142 \in \mf{S}_5$. Such permutations are \emph{simple permutations}, and appear throughout the literature of permutations and permutation patterns (see, for example, \cite{albert atkinson, brignall}). Let
$$\simple(n)$$
denote the number of simple permutations in $\mf{S}_n$, yielding the sequence $1, 2, 0, 2, 6, 46, 338, \ldots$ \cite[A111111]{oeis}. Note that $\simple(3) = 0$, whereas $\simple(n) > 0$ for all other positive integers $n$.

As shown in Example~\ref{ex:43187562}, the intervals of a permutation need not be disjoint. The relationships among the intervals of a permutation have a natural poset structure, and this is what we study in this work. The reader is referred to \cite{ec1} for basic poset terminology.

\begin{definition}\label{defn:interval poset}
The \emph{interval poset} of $w \in \mf{S}_n$ is the poset $\poset(w)$ whose elements are the nonempty intervals of $w$ and whose order relations are defined by set inclusion. The \emph{closed interval poset} $\hatposet(w)$ is obtained from $\poset(w)$ by adjoining a minimum element $\widehat{0}$, which we think of as representing the empty interval.
\end{definition}

The embedding of this poset matters, and the vertices will have a canonical labeling: minimal elements will be labeled $1, 2, 3, \ldots$ from left to right, and non-minimal elements will be labeled by the collection (interval) of labels below them in the poset.

Put another way, $\hatposet(w)$ is the induced subposet of the boolean algebra on $[n]$ formed by the (possibly empty) intervals of $w$. The interval poset for the permutation $43187562$ is shown in Figure~\ref{fig:ex poset(43187562)}, and the poset for $123$ appears in Figure~\ref{fig:ex poset(123)}. As the figures indicate, interval posets are not necessarily trees.
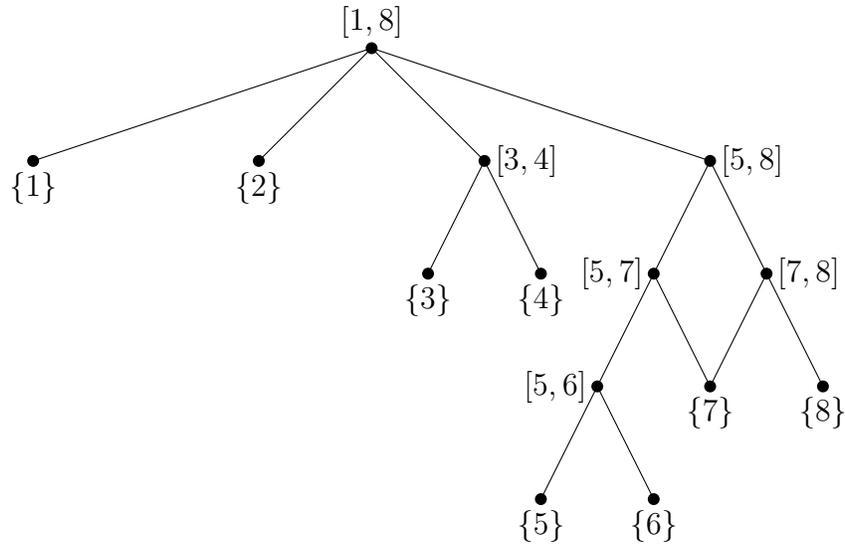
\begin{figure}[htbp]
\begin{tikzpicture}[scale=.75]
\fill (0,12) circle (3pt) coordinate (18);
\fill (-6,10) circle (3pt) coordinate (1);
\fill (-2,10) circle (3pt) coordinate (2);
\fill (2,10) circle (3pt) coordinate (34);
\fill (6,10) circle (3pt) coordinate (58);
\fill (1,8) circle (3pt) coordinate (3);
\fill (3,8) circle (3pt) coordinate (4);
\fill (5,8) circle (3pt) coordinate (57);
\fill (7,8) circle (3pt) coordinate (78);
\fill (4,6) circle (3pt) coordinate (56);
\fill (6,6) circle (3pt) coordinate (7);
\fill (8,6) circle (3pt) coordinate (8);
\fill (3,4) circle (3pt) coordinate (5);
\fill (5,4) circle (3pt) coordinate (6);
\foreach \x in {1,2,34,58} {\draw (18) -- (\x);}
\foreach \x in {3,4} {\draw (34) -- (\x);}
\foreach \x in {57,78} {\draw (58) -- (\x);}
\foreach \x in {56,7} {\draw (57) -- (\x);}
\foreach \x in {7,8} {\draw (78) -- (\x);}
\foreach \x in {5,6} {\draw (56) -- (\x);}
\foreach \x in {1,2,3,4,5,6,7,8} {\draw (\x) node[below] {$\{\x\}$};}
\draw (56) node[left] {$[5,6]$};
\draw (57) node[left] {$[5,7]$};
\draw (78) node[right] {$[7,8]$};
\draw (58) node[right] {$[5,8]$};
\draw (34) node[right] {$[3,4]$};
\draw (18) node[above] {$[1,8]$};
\end{tikzpicture}
\caption{The interval poset $\poset(43187562)$.} \label{fig:ex poset(43187562)}
\end{figure}

\begin{figure}[htbp]
\begin{tikzpicture}[scale=.75]
\fill (0,4) circle (3pt) coordinate (13);
\fill (-1,2) circle (3pt) coordinate (12);
\fill (1,2) circle (3pt) coordinate (23);
\fill (-2,0) circle (3pt) coordinate (1);
\fill (0,0) circle (3pt) coordinate (2);
\fill (2,0) circle (3pt) coordinate (3);
\draw (-2,0) -- (0,4) -- (2,0);
\draw (-1,2) -- (0,0) -- (1,2);
\draw (13) node[above] {$[1,3]$};
\draw (1) node[below] {$\{1\}$};
\draw (2) node[below] {$\{2\}$};
\draw (3) node[below] {$\{3\}$};
\draw (12) node[left] {$[1,2]$};
\draw (23) node[right] {$[2,3]$};
\end{tikzpicture}
\caption{The interval poset $\poset(123)$.} \label{fig:ex poset(123)}
\end{figure}
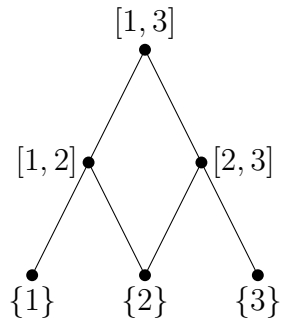

The interval poset of a permutation is closely related to the \emph{substitution decomposition} of Albert and Atkinson \cite{albert atkinson}, as we will discuss in the next section.
 
The goal of this paper is to study the interval poset and closed interval poset of a permutation. We lay out terminology and basic facts about interval posets in Section~\ref{sec:basic facts}. Many of these follow immediately from the definition of the poset. In Section~\ref{sec:structural}, we answer deeper questions about the structure of these posets (Theorems~\ref{thm:planar}, \ref{thm:lattice}, \ref{thm:modular}, and~\ref{thm:distributive}). 
Section~\ref{sec:characterizing} turns to questions of characterization; namely, which posets are of the form $\poset(w)$ for some $w$ (Theorem~\ref{thm:interval poset characterization}). 
In Section~\ref{sec:enumerate}, we enumerate the permutations $w$ satisfying $\poset(w) = \mathcal{P}$ for a given $\mathcal{P}$ (Theorem~\ref{thm:number of permutations for a poset}). 
Special families of interval posets are studied in Section~\ref{sec:special families}, with characterizations in terms of pattern avoidance (Theorems~\ref{thm:interval trees} and~\ref{thm:binary} and Corollary~\ref{cor:binary trees}) and enumerations (Corollaries~\ref{cor:binary enumeration} and~\ref{cor:binary tree enumeration}). 
We close the paper with proposed directions for further research.

\section{Preliminaries and the substitution decomposition}\label{sec:basic facts}

Fix a permutation $w \in \mf{S}_n$. As is typical with permutations, we will frequently refer to $\{1, \ldots, n\}$ as ``letters'' of the permutation. The permutation $1 \in \mf{S}_1$ is of limited intrigue, and so unless specified otherwise, we will always assume that permutations have at least two letters and interval posets have at least two minimal elements. Both $\poset(w)$ and $\hatposet(w)$ have a unique maximal element: the interval $[1,n]$. Both posets are finite, with $|\poset(w)| = |\hatposet(w)| - 1 \le \binom{n}{2}-1$, since intervals must be consecutive. Following \cite{ec1}, we way that the \emph{rank} of a poset is the maximum chain length in the poset. The maximal chains in our posets might not all have the same length, and the rank of $\poset(w)$ can be anything between $1$ and $n-1$ (with the exception that the rank is always $2$ when $n=3$). The rank of $\hatposet(w)$ is one more than the rank of $\poset(w)$. It follows from the definitions that each element in $\poset(w)$ represents an interval of consecutive integers, and no element of $\poset(w)$ covers exactly one element. Principal order ideals, also called down-sets, will be a key object in this work.

\begin{definition}
Given a poset $P$ and an element $x \in P$, the \emph{principal order ideal} of $x$ is the collection of elements less than or equal to $x$; that is, it is the set $\{y \in P : y \le x\}$.
\end{definition}

The minimal elements in $\poset(w)$ represent the letters permuted by $w$, and an arbitrary element in $\poset(w)$ represents the interval formed by the union of all minimal elements in its principal order ideal.

By definition, every simple permutation in $\mf{S}_n$ has the same interval poset: a maximal element covering $n$ minimal elements.

The next result is straightforward, but useful enough to warrant writing down.

\begin{lemma}\label{lem:poi}
The principal order ideal of an interval $I \in \poset(w)$ is isomorphic to the interval poset for the permutation formed by restricting the word $w$ to the letters of $I$; similarly for principal order ideals in the closed interval poset.
\end{lemma}

For example, the principal order ideal of $[5,8] \in \poset(43187562)$, depicted in Figure~\ref{fig:ex poset(43187562)}, is isomorphic to the interval poset of the permutation (that is, the subword) $8756$; equivalently, to the interval poset of its order isomorphic permutation $4312$.

Another feature of interval posets that follows directly from definitions will also be critical to some of our later arguments.

\begin{lemma}\label{lem:isolated cover}
An interval $H \in \poset(w)$ is covered by exactly one interval $I$ if and only if for every interval $J \supset H$, we also have $J \supseteq I$.
\end{lemma}

A familiar involution on permutations is relevant to this work.

\begin{definition}
The \emph{reverse} of the permutation $w = w(1)\cdots w(n)$ is the permutation
$$w^R := w(n) \cdots w(1).$$
\end{definition}

Because intervals depend on being consecutive, with no directional preference, and because we take ``intervals'' in a permutation to be intervals of values, a permutation and its reverse have the same intervals. Other permutation symmetries (such as inverse and complement) preserve the structure of the interval poset, but not generally the intervals themselves.

\begin{lemma}\label{lem:reverse}
$\poset(w) = \poset(w^R)$ for all $w$.
\end{lemma}

Throughout this work, we will prefer to draw interval posets in a particular way.

\begin{definition}\label{defn:increasing intervals}
If $I_1, I_2, \ldots, I_r$ are non-nesting intervals so that $\min(I_j) < \min(I_{j+1})$ for all $j$ (equivalently, $\max(I_j) < \max(I_{j+1})$), then the sequence $I_1, \ldots, I_r$ is in \emph{increasing order}.
\end{definition}

No two intervals in an antichain are nesting; that is, neither interval is a subset of the other.

\begin{definition}\label{defn:canonical embedding}
The \emph{canonical} Hasse diagram of $\poset(w)$ is drawn so that the (antichain of) elements of a fixed depth from the maximal element $[1,n]$ appear at the same height, and in increasing order from left to right across the poset. The canonical Hasse diagram of $\hatposet(w)$ is obtained by adjoining a minimum element to that diagram of $\poset(w)$.
\end{definition}

Figure~\ref{fig:ex poset(43187562)} depicts the canonical Hasse diagram of $\poset(43187562)$.

We will always draw the canonical Hasse diagrams of $\poset(w)$ and $\hatposet(w)$. Noting that the minimal elements of $\poset(w)$ (atoms of $\hatposet(w)$) are the $1$-element intervals $\{i\}$ for $i \in [n]$, we can omit the element labels in a canonical Hasse diagram. It will sometimes be useful to think of Hasse diagrams as directed graphs, with the maximal element as the root, and we will use this language (e.g., ``tree,'' ``child'') accordingly.

We close this section by recalling a standard operation on permutations.

\begin{definition}\label{defn:inflation}
Fix $w \in \mf{S}_n$ and $p_i \in \mf{S}_{m_i}$ for $i \in [1,n]$. The \emph{inflation} of $w$ by $p_1, \ldots, p_n$ is the permutation $w[p_1,\ldots, p_n] \in \mf{S}_{\sum m_i}$ defined by replacing each $w(i)$ by an interval that is order isomorphic to $p_i$, so that the letters in the segment replacing $w(i)$ are all larger than the letters in the segment replacing $w(j)$ if and only if $w(i) > w(j)$. When there is a $j$ such that $m_i = 1$ for all $i \neq j$, we refer to this as \emph{inflating the $w(j)$ in $w$ by $p_j$}.
\end{definition}

\begin{example}
$3142[21,1,4312,1] = 43187562$. Inflating the $3$ in $3176452$ by $21$ produces that same permutation: $43187562$.
\end{example}

This brings us to the so-called \emph{substitution decomposition} of Albert and Atkinson (see \cite{albert atkinson} and also \cite{brignall}).

\begin{proposition}[{\cite{albert atkinson}}]\label{prop:substitution decomposition}
For any $n \ge 2$ and any $w \in \mf{S}_n$, there is a unique simple permutation $v \in \mf{S}_k$ for which $w = v[p_1,\ldots,p_k]$. Moreover, if $k \ge 4$ (meaning, in particular, that $v \neq 12,21$), then the permutations $p_1,\ldots,p_k$ are uniquely determined.
\end{proposition}

A permutation that can be written as $12[p,q] = p\oplus q$ is \emph{sum decomposable}, and \emph{sum indecomposable} otherwise. A permutation that can be written as $21[p,q] = p\ominus q$ is \emph{skew decomposable}, and \emph{skew indecomposable} otherwise. From Lemma~\ref{lem:reverse}, we know that sum decomposable permutations and skew decomposable permutations have the same interval posets for the same pair $(p,q)$.

Proposition~\ref{prop:substitution decomposition}, and the partitioning of $\mf{S}_n$ that we can make as a result, lay the groundwork for interval posets.

\begin{corollary}\label{cor:permutations are inflations or sums}
Each permutation $w$ falls into exactly one of the following categories:
\begin{itemize}
\item $w = v[p_1,\ldots,p_k]$ where $v$ is simple and $k \ge 4$,
\item $w$ is sum decomposable, in which case we can write $w$ as a maximal direct sum of sum indecomposable permutations $w = p_1 \oplus \cdots \oplus p_k$, or
\item $w$ is skew decomposable, in which case we can write $w$ as a maximal skew sum of skew indecomposable permutations $w = p_1 \ominus \cdots \ominus p_k$.
\end{itemize}
\end{corollary}

\section{Structural properties of interval posets}\label{sec:structural}

In this section, we describe the structural properties of the interval poset $\poset(w)$ of a permutation, and of $\hatposet(w)$ where relevant. We begin with a notion of planarity, such as was formalized by Lakser and presented by Quackenbush in \cite{quackenbush}: a poset is \emph{planar} if it has a planar Hasse diagram. The interval poset depicted in Figure~\ref{fig:ex poset(43187562)} is planar, and indeed this is no coincidence. Unlike the other structural properties that we explore below, planarity depends on finding a specific embedding of the poset in the plane. In fact, we can explicitly describe an embedding of the poset $\poset(w)$ that always produces a planar diagram.

We will use a straightforward property of intervals, here and subsequently, from 1980.

\begin{lemma}[F\"oldes \cite{foldes}]\label{lem:interval intersection}
Suppose that $I$ and $J$ are intervals of the permutation $w$. Then $I \cap J$ and $I \setminus J$ are also both intervals of $w$. Moreover, if $I \cup J \neq \emptyset$, then $I \cap J$ is an interval of $w$.
\end{lemma}

\begin{theorem}\label{thm:planar}
For any permutation $w$, the posets $\poset(w)$ and $\hatposet(w)$ are planar.
\end{theorem}

\begin{proof}
It is enough to prove the result for $\poset(w)$. Consider the canonical Hasse diagram of $\poset(w)$, and suppose that it is not planar. Then there must be a configuration of covering relations such as this.
$$
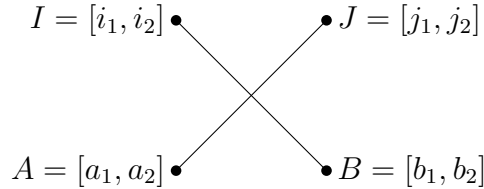
\begin{tikzpicture}
\foreach \x in {(0,0), (2,0), (0,2), (2,2)} {\fill \x circle (2pt);}
\draw (0,0) -- (2,2);
\draw (0,2) -- (2,0);
\draw (0,0) node[left] {$A = [a_1,a_2]$};
\draw (0,2) node[left] {$I = [i_1,i_2]$};
\draw (2,0) node[right] {$B = [b_1,b_2]$};
\draw (2,2) node[right] {$J = [j_1,j_2]$};
\end{tikzpicture}$$
Because $A \subset J$ and $B \subset I$, we have $j_1 \le a_1 \le a_2 \le j_2$ and $i_1 \le b_1 \le b_2 \le i_2$. The embedding of the canonical Hasse diagram means that the set $\{A,B\}$ is an antichain, and $a_1 < b_1$ and $a_2 < b_2$. Similarly, $i_1 < j_1$ and $i_2 <j_2$. Putting this information together, we have
$$i_1 < j_1 \le a_1 < b_1 \le b_2 \le i_2 < j_2,$$
with $a_2 \in [a_1,b_2)$. By Lemma~\ref{lem:interval intersection}, we have $I \cap J = [j_1,i_2] \in \poset(w)$. But then $B \subset I$ is not a covering relation, because $B \subsetneq [j_1,i_2] \subsetneq I$, which is a contradiction.
\end{proof}

We now examine whether $\hatposet(w)$ belongs to several well-known poset classes. Once again, the reader is referred to \cite{ec1} for definitions and further details.

We start with lattices. In fact, this result follows from Theorem~\ref{thm:planar} and a theorem of Lakser (see \cite{quackenbush}), but we present a proof specific to the interval poset here, to highlight how the lattice structure works in this context.

\begin{theorem}\label{thm:lattice}
For any permutation $w$, the poset $\hatposet(w)$ is a lattice.
\end{theorem}

\begin{proof}
Consider two intervals $I$ and $J$ in $w$. We will construct their unique greatest lower bound (meet, denoted $\wedge$) and unique least upper bound (join, denoted $\vee$). 

Consider $I \cap J$. Anything less than both $I$ and $J$ in $\hatposet(w)$ must also be less than $I \cap J$ in $\hatposet(w)$. Because $I$ and $J$ are intervals, their intersection must also be an interval. If this set is empty, then $I \wedge J = \widehat{0}$. Suppose, now, that $I \cap J$ is not empty. As argued in the proof of Theorem~\ref{thm:planar}, this $I \cap J$ is an interval in $w$. Therefore, $I \cap J \in \poset(w)$ is the meet of $I$ and $J$.

This shows that $\hatposet(w)$ is a meet semi-lattice. Because the poset $\hatposet(w)$ also has a maximal element, it is in fact a lattice \cite[Proposition 3.3.1]{ec1}.
\end{proof}

The lattice $\hatposet(w)$ is clearly atomic, by construction, but it is not always complemented as we see in the following example.

\begin{example}
Let $w = 123$. There is no element $p \in \hatposet(w)$ such that $p \vee \{2\} = [1,3]$ and $p \wedge \{2\} = \widehat{0}$. The poset $\hatposet(w)$ is shown in Figure~\ref{fig:poset(123)}, with $\{2\}$ circled in that figure. Note the relationship between that figure and Figure~\ref{fig:ex poset(123)}.
\end{example}

\begin{figure}[htbp]
\begin{tikzpicture}[scale=.75]
\fill (0,0) circle (3pt) coordinate (0);
\fill (-2,2) circle (3pt) coordinate (1);
\fill (0,2) circle (3pt) coordinate (2);
\fill (2,2) circle (3pt) coordinate (3);
\fill (-1,4) circle (3pt) coordinate (12);
\fill (1,4) circle (3pt) coordinate (23);
\fill (0,6) circle (3pt) coordinate (123);
\foreach \x in {1,2,3} {\draw (0) -- (\x);}
\foreach \x in {12,23} {\draw (123) -- (\x);}
\foreach \x in {1,2} {\draw (12) -- (\x);}
\foreach \x in {2,3} {\draw (23) -- (\x);}
\draw (2) circle (5pt);
\end{tikzpicture}
\caption{The lattice $\hatposet(123)$, with canonical Hasse diagram, which is not complemented. The element $\{2\}$ has been circled.}\label{fig:poset(123)}
\end{figure}
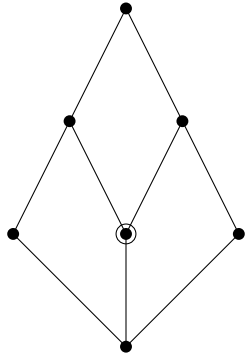

We now turn to classifying modular interval posets for permutations; that is, those interval posets in which $a \le b$ implies $a \vee (x \wedge b) = (a \vee x) \wedge b$ for every $x$. In fact, these are quite a small class. We will use the fact that modular posets can be characterized by avoidance of the sublattice $N_5$, known as the ``pentagon lattice.''

\begin{theorem}\label{thm:modular}
An interval poset $\hatposet(w)$ is modular if and only if $w$ is a simple permutation.
\end{theorem}

\begin{proof}
First consider $w \in \mf{S}_n$ for which $\hatposet(w)$ is modular. Suppose that $w$ has a proper interval $I$ containing $1$. Because $I$ is proper, $n \not\in I$. Then the poset $\{\widehat{0}, \{1\}, \{n\}, I, I \vee \{n\} = [1,n] \}$ is isomorphic to the pentagon lattice $N_5$, meaning that $\hatposet(w)$ would not be modular. Therefore there can be no such interval $I$. A similar argument shows that $w$ has no proper intervals containing $n$. Now suppose that $w$ has a proper interval $I$ containing $j$ for some $j \in [2,n-1]$ (and necessarily $1 \not\in I$). Then the poset $\{\widehat{0}, \{1\}, \{j\}, I, \{1\} \vee I = [1,n]\}$ is isomorphic to $N_5$, and so $\hatposet(w)$ would not be modular. Thus there are no proper intervals in $w$, so $w$ is simple. 

Now suppose that $w \in \mf{S}_n$ is simple. The poset $\hatposet(w)$ has rank $2$: a minimal element, a maximal element, and $n$ atoms/coatoms. This certainly contains no sublattice isomorphic to $N_5$, and therefore it is modular.
\end{proof}

Before classifying distributive interval posets for permutations, we need some terminology.

\begin{definition}\label{defn:fruitful}
Fix a permutation $w$ and consider its interval poset $\poset(w)$. An element $I \in \hatposet(w)$ is  \emph{fruitful} if $I$ covers at least $3$ elements. 
\end{definition}

Fruitful elements highlight a particular phenomenon among the intervals of a permutation, giving a restatement of the substitution decomposition \cite{albert atkinson}.

\begin{lemma}[cf. \cite{albert atkinson}]\label{lem:fruitful elements are simple inflations}
Fix a fruitful element $I \in \poset(w)$, covering $\{H_1, \ldots, H_k\}$ for some $k \ge 3$. The interval $I$, as it appears in the one-line notation for $w$, is order isomorphic to the inflation of a simple permutation $v \in \mf{S}_k$. Moreover, the intervals $\{H_i\}$ are pairwise disjoint.
\end{lemma}

\begin{proof}
The maximal proper intervals in $I$ are $H_1, \ldots, H_k$, which we can assume to be listed in increasing order. If $H_j$ and $H_{j+1}$ intersect nontrivially, then $H_j \cup H_{j+1}$ would also be an interval in $I$, contradicting either maximality of the intervals or the value of $k$. Thus all elements of $H_j$ are less than all elements of $H_{j+1}$. Consider $w|_I$, the word obtained by restricting the one-line notation of $w$ to the letters in $I$. For each $j$, replace the entire interval $H_j$ in $w|_I$ by the letter $j$, and let the resulting permutation of $[k]$ be called $v$. Then $w|_I$ is the inflation $v[H_1, \ldots, H_k]$. The permutation $v$ is simple because it has no proper intervals.
\end{proof}

This result gives an immediate restriction on the possible structure of $\poset(w)$, because $\mf{S}_3$ has no simple permutations. Thus, in fact, we could revise Definition~\ref{defn:fruitful} to say that an element is fruitful if it covers at least four elements. We will use that bound henceforth.

\begin{corollary}\label{cor:no 3 kids}
No element of $\poset(w)$ covers exactly three elements.
\end{corollary}

Because simple permutations have no proper intervals, the intervals $H_i$ and $H_{i+1}$ cannot appear consecutively in $w$, yielding the following handy result.

\begin{lemma}\label{lem:fruitful segments have just one parent and are disjoint}
Fix a permutation $w$. Suppose that $I \in \poset(w)$ is a fruitful element, covering $\{H_1, \ldots, H_k\}$ for $k \ge 4$. Then $I$ is the only element to cover each $H_i$.
\end{lemma}

\begin{proof}
Without loss of generality, assume that the sequence $H_1, \ldots, H_k$ is in increasing order. By Theorem~\ref{thm:planar}, the only intervals from this list that could be covered by something in addition to $I$ are $H_1$ and $H_k$. By Lemma~\ref{lem:fruitful elements are simple inflations}, the letters of $I$ as they appear in $w$ are isomorphic to the inflation of a simple permutation in $\mf{S}_k$ for $k \ge 4$, and neither $1$ nor $k$ can appear at either end of such a permutation. Thus no element can cover $H_1$ without also covering some other $H_i$, which will violate Theorem~\ref{thm:planar}. That the intervals $\{H_i\}$ are pairwise disjoint follows as argued in the proof of Lemma~\ref{lem:fruitful elements are simple inflations}.
\end{proof}

\begin{proposition}\label{prop:fruitful elements aren't distributive}
If $\hatposet(w)$ has a fruitful element, then it is not a distributive lattice.
\end{proposition}

\begin{proof}
Let $I \in \hatposet(w)$ be a fruitful element, and $H_1, H_2, H_3$ be three of the elements that it covers. By Lemma~\ref{lem:fruitful elements are simple inflations}, the intervals $H_1$, $H_2$, and $H_3$ are pairwise disjoint. Thus, as discussed in Theorem~\ref{thm:lattice}, their pairwise meets are $\widehat{0}$.
Then $\{I, H_1, H_2, H_3, \widehat{0}\}$ is isomorphic to the diamond lattice $M_3$, and so $\hatposet(w)$ is not distributive.
\end{proof}

The configuration referenced in the proof of Proposition~\ref{prop:fruitful elements aren't distributive} is depicted in Figure~\ref{fig:diamond lattice}.

\begin{figure}[htbp]
\begin{tikzpicture}[scale=.75]
\fill (0,0) circle (3pt) coordinate (0) node[below] {$\widehat{0}$};
\fill (-2,2) circle (3pt) coordinate (1);
\fill (0,2) circle (3pt) coordinate (2);
\fill (2,2) circle (3pt) coordinate (3);
\fill(.5,4) circle (3pt) coordinate (I) node[above] {$I$};
\draw (I) -- (2,3);
\draw (I) -- (2.5,3);
\draw (3.25,3) node {$\ldots$};
\foreach \x in {1,2,3} {\draw[dashed] (0) -- (\x); \draw (I) -- (\x);}
\end{tikzpicture}
\caption{The diamond lattice constructed in the proof of Proposition~\ref{prop:fruitful elements aren't distributive}.}\label{fig:diamond lattice}
\end{figure}
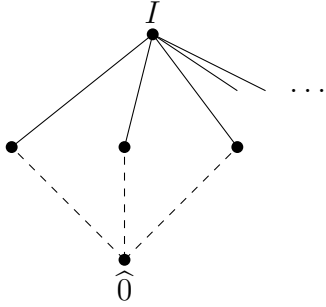

We can now combine Theorem~\ref{thm:modular} and Proposition~\ref{prop:fruitful elements aren't distributive} to (quite strictly!) characterize distributive interval posets for permutations. As with modular lattices, we will use a sublattice characterization: distributive lattices are exactly those avoiding the diamond lattice $M_3$ and the pentagon lattice $N_5$. The result itself is almost disappointingly restrictive, but we include it as a theorem in analogy to Theorems~\ref{thm:planar}, \ref{thm:lattice}, and~\ref{thm:modular}.

\begin{theorem}\label{thm:distributive}
An interval poset $\hatposet(w)$ is distributive if and only if $w \in \mf{S}_1 \cup \mf{S}_2$.
\end{theorem}

\begin{proof}
Suppose that $w \in \mf{S}_n$ and $\hatposet(w)$ is distributive. Distributive lattices are modular, so $w$ must be simple, by Theorem~\ref{thm:modular}. To avoid a sublattice isomorphic to the diamond lattice $M_3$, we need $n < 3$.

It is straightforward to check that $\hatposet(w)$ is distributive for all $w \in \mf{S}_1 \cup \mf{S}_2$.
\end{proof}

\section{Characterizing interval posets}\label{sec:characterizing}

The goal of this section is, in a sense, to learn how to invert the map $w \mapsto \poset(w)$. To that end, we define the following (possibly empty) set.

\begin{definition}\label{defn:interval poset set}
Fix a poset $P$. The permutations whose interval poset is $P$ are the \emph{interval generators} of $P$, and the set of interval generators of $P$ is denoted
$$\interval(P) := \{w : \poset(w) = P\}.$$
\end{definition}

The classification presented in Corollary~\ref{cor:permutations are inflations or sums} gives almost everything necessary for characterizing these posets.

\begin{definition}\label{defn:poset types}
A \emph{dual claw poset} $\Lambda_k$ has a unique maximal element, and $k \ge 4$ minimal elements that are all covered by that maximal element. We use the term \emph{argyle poset} to describe the interval poset of $12\cdots k$ for some $k \ge 2$. A \emph{binary tree poset} is a poset whose Hasse diagram is a tree in which every node has $2$ children except for the minimal elements, and all maximal chains have the same length.
\end{definition}

\begin{figure}[htbp]
\begin{tikzpicture}[scale=.75]
\fill (21.5,6) circle (3pt) coordinate (o);
\fill (20,4) circle (3pt) coordinate (L);
\fill (23,4) circle (3pt) coordinate (R);
\fill (19,2) circle (3pt) coordinate (LL);
\fill (21,2) circle (3pt) coordinate (LR);
\fill (22,2) circle (3pt) coordinate (RL);
\fill (24,2) circle (3pt) coordinate (RR);
\fill (21,0) circle (3pt) coordinate (RLL);
\fill (23,0) circle (3pt) coordinate (RLR);
\draw (LL) -- (L) -- (o) -- (R) -- (RR);
\draw (LR) -- (L);
\draw (RLL) -- (RL) -- (R);
\draw (RLR) -- (RL);
\fill (7,4) circle (3pt) coordinate (0);
\foreach \x in {5,6,7,8,9} {\fill (\x,2) circle (3pt); \draw (\x,2) -- (0);}
\draw (11,0) -- (14,6) -- (17,0);
\draw (12,2) -- (13,0) -- (15,4);
\draw (13,4) -- (15,0) -- (16,2);
\foreach \x in {11,13,15,17} {\fill (\x,0) circle (3pt);}
\foreach \x in {12,14,16} {\fill (\x,2) circle (3pt);}
\foreach \x in {13,15} {\fill (\x,4) circle (3pt);}
\fill (14,6) circle (3pt);
\end{tikzpicture}
\caption{From left to right: the dual claw poset $\Lambda_5$, an argyle poset, and a binary tree poset.}
\label{fig:poset types}
\end{figure}
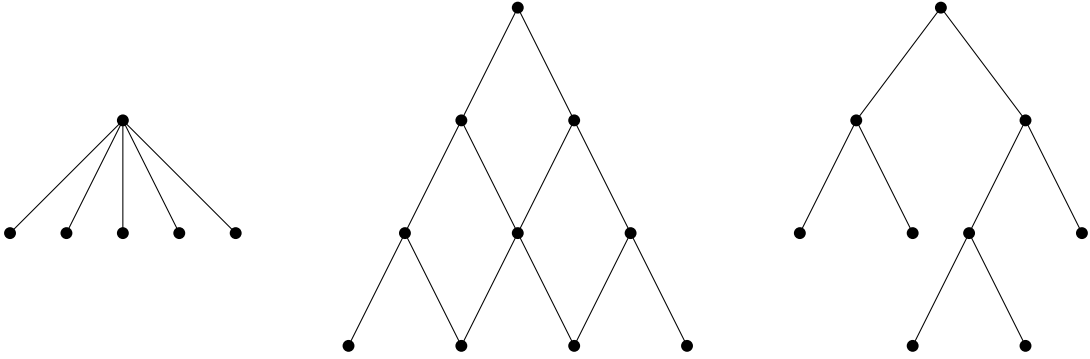

These three classes of posets, illustrated in Figure~\ref{fig:poset types}, are, in fact, interval posets. Moreover, we can describe the generating set $\interval(P)$ in each case.

\begin{proposition}\label{prop:fruitful trees are interval posets}
For any $k \ge 4$, the dual claw poset $\Lambda_k$ is the interval poset for exactly $\simple(k)$ permutations: the simple permutations in $\mf{S}_k$.
\end{proposition}

\begin{proof}
The dual claw poset $\Lambda_k$ is the interval poset of a permutation $w \in \mf{S}_k$ if and only if there are no proper intervals in $w$. This is the definition of being a simple permutation.
\end{proof}

Showing that argyle posets are interval posets is similarly easy.

\begin{proposition}\label{prop:argyle posets are interval posets}
The argyle poset with $k \ge 2$ minimal elements is the interval poset for exactly two permutations: $12\cdots k$ and $(12\cdots k)^R = k \cdots 21$.
\end{proposition} 

\begin{proof}
An argyle poset with $k$ minimal elements is the interval poset of a permutation $w \in \mf{S}_k$ if and only if each $j$ and $j+1$ appear consecutively. The only permutations with this property are the identity and its reverse.
\end{proof}

Binary trees arise from the separable permutations that decompose via alternating direct sums and skew sums into exactly two components at each step. The embedding matters here, which was not an issue with $\Lambda_k$ or argyle posets.

\begin{definition}
Fix a binary tree poset $P$ and its embedding, with $k\ge 2$ minimal nodes labeled $1, \ldots, k$. Those nodes should be labeled in increasing order according to a depth-first search that always chooses left edges before right edges. An \emph{alternating depth-first search} of $P$ is a depth-first search that choose left edges before right edges at all choices that are of even (respectively, odd) distance from the maximal element, and right edges before left edges at all choices that are odd (resp., even) distance from the maximal element. The \emph{ADFS} words of $P$ are the two words formed by minimal element labels in the order seen when performing alternating depth-first searches of $P$.
\end{definition}

\begin{example}
The ADFS words of the binary tree poset in Figure~\ref{fig:poset types} are $21534$ and $43512$.
\end{example}

The ADFS words of a binary tree poset are reverses of each other.

\begin{proposition}\label{prop:binary trees are interval posets}
A binary tree poset with $k \ge 2$ minimal elements is the canonical Hasse diagram of the interval posets for exactly two permutations: the ADFS words of the poset.
\end{proposition}

\begin{proof}
This result follows more or less immediately from the fact that each internal vertex of the poset corresponds to a permutation of the form $p\oplus q$ or $p\ominus q$, where $p$ and $q$ are, respectively, sum or skew indecomposable. There are two choices for the permutation, depending on whether the maximal element represents a direct sum or a skew sum.
\end{proof}

Lemma~\ref{lem:poi} suggests using a recursive inflation to describe exactly which posets occur as $\poset(w)$ for some $w$, mimicking the successive inflation of permutation entries.

\begin{theorem}\label{thm:interval poset characterization}
Interval posets are exactly those that can be constructed by starting with the $1$-element poset, and recursively replacing minimal elements with dual claw posets, argyle posets, or binary tree posets.
\end{theorem}

\begin{proof}
We point out that if $w$ is the inflation of $v[p_1,\ldots,p_k]$ for $k \ge 4$, then $\poset(w)$ consists of a single maximal element covering the disjoint interval posets of $p_1,\ldots,p_k$ in some order (determined by $v^{-1}$). Otherwise, suppose without loss of generality that $w = p_1 \oplus \cdots \oplus p_k$ where the $p_i$ are sum indecomposable. Then there are two maximal proper intervals in $w$: those formed by $p_1 \oplus \cdots \oplus p_{k-1}$ and $p_2 \oplus \cdots \oplus p_k$. Hence $\poset(w)$ can be built by replacing minimal elements in an argyle poset.
\end{proof}

We demonstrate Theorem~\ref{thm:interval poset characterization} using the permutation $43187562$ whose interval poset was depicted in Figure~\ref{fig:ex poset(43187562)}.

\begin{example}\label{ex:building the tree}
$\poset(43187562)$ can be built by noting that
\begin{align*}
43187562 &= 3142[21,1,4312,1]\\
&= 3142[21,1,321[1,1,12],1],
\end{align*}
as shown in Figure~\ref{fig:building poset(43187562)}. 
\end{example}

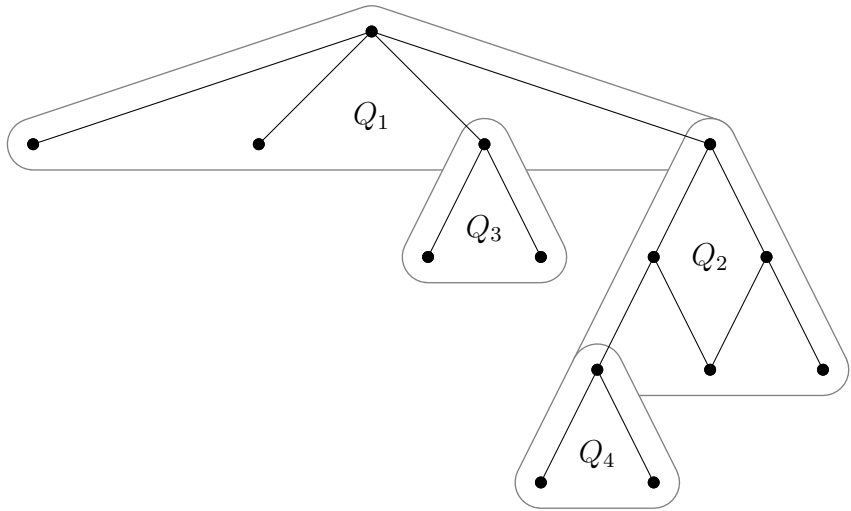
\begin{figure}[htbp]
\begin{tikzpicture}[scale=.75]
\fill (0,12) circle (3pt) coordinate (18);
\fill (-6,10) circle (3pt) coordinate (1);
\fill (-2,10) circle (3pt) coordinate (2);
\fill (2,10) circle (3pt) coordinate (34);
\fill (6,10) circle (3pt) coordinate (58);
\fill (1,8) circle (3pt) coordinate (3);
\fill (3,8) circle (3pt) coordinate (4);
\fill (5,8) circle (3pt) coordinate (57);
\fill (7,8) circle (3pt) coordinate (78);
\fill (4,6) circle (3pt) coordinate (56);
\fill (6,6) circle (3pt) coordinate (7);
\fill (8,6) circle (3pt) coordinate (8);
\fill (3,4) circle (3pt) coordinate (5);
\fill (5,4) circle (3pt) coordinate (6);
\foreach \x in {1,2,34,58} {\draw (18) -- (\x);}
\foreach \x in {3,4} {\draw (34) -- (\x);}
\foreach \x in {57,78} {\draw (58) -- (\x);}
\foreach \x in {56,7} {\draw (57) -- (\x);}
\foreach \x in {7,8} {\draw (78) -- (\x);}
\foreach \x in {5,6} {\draw (56) -- (\x);}
\foreach \x in {18,1,2,34,58,3,4,57,78,56,7,8,5,6} {\fill (\x) circle (3pt);}
\drawpolygon 18,1,58;
\drawpolygon 34,3,4;
\drawpolygon 58,56,8;
\drawpolygon 56,5,6;
\draw (0,10.5) node {$Q_1$};
\draw (6,8) node {$Q_2$};
\draw (2,8.5) node {$Q_3$};
\draw (4,4.5) node {$Q_4$};
\end{tikzpicture}
\caption{The poset $\poset(43187562)$ is constructed in four steps. The labeling of those steps refers to Example~\ref{ex:building the poset}.}
\label{fig:building poset(43187562)}
\end{figure}

\section{Enumerative properties of interval posets}\label{sec:enumerate}

Lemma~\ref{lem:reverse} showed that a permutation and its reverse have the same interval poset. Therefore, there are at most $n!/2$ interval posets with $n$ minimal elements. In fact, that is an overcount when $n\ge 4$, because simple permutations of a given size have the same poset. The enumeration below comes from infusing Theorem~\ref{thm:interval poset characterization} with Propositions~\ref{prop:fruitful trees are interval posets}, \ref{prop:argyle posets are interval posets}, and~\ref{prop:binary trees are interval posets}, and Lemma~\ref{lem:isolated cover}.

\begin{theorem}\label{thm:number of permutations for a poset}
Fix an interval poset $P$. Let $Q_0$ be the $1$-element poset, and let $Q_1, Q_2, \ldots$ be the dual claw, argyle, and binary tree posets that recursively replace minimal elements in the construction of an interval poset $P$, as described in Theorem~\ref{thm:interval poset characterization}. Then
$$\left| \interval(P) \right| = \prod_{i=0} \left| \interval(Q_i) \right|^{\varepsilon_i},$$
where
$$\varepsilon_i = \begin{cases}
1 & \text{if $i=0$,}\\
1 & \text{if $Q_i$ is a dual claw,}\\
1 & \text{if $Q_i$ replaces the child of a dual claw, and}\\
0 & \text{otherwise.}
\end{cases}$$
\end{theorem}

Each $Q_i$ in the statement of Theorem~\ref{thm:number of permutations for a poset} has three possibilities, and $\left|\interval(Q_i)\right|$ was computed for each of these in Propositions~\ref{prop:fruitful trees are interval posets}, \ref{prop:argyle posets are interval posets}, and~\ref{prop:binary trees are interval posets}:
$$\left| \interval(Q_i) \right| = \begin{cases}
1 & \text{if } i = 0,\\
\simple(k) & \text{if $i > 0$ and $Q_i$ is $\Lambda_k$, and}\\
2 & \text{otherwise}.
\end{cases}$$
Although there may be multiple ways to construct a given poset, this does not affect the enumeration of the interval generators: one cannot alter the position or size of any dual claw posets that appear.

\begin{example}\label{ex:building the poset}
Let $P$ be the interval poset depicted in Figure~\ref{fig:ex poset(43187562)} and constructed as in Figure~\ref{fig:building poset(43187562)}:
$$\left| \interval(P) \right| = 1^1 \cdot 2^1 \cdot 2^1 \cdot 2^1 \cdot 2^0 = 8.$$
Indeed, as outlined in Example~\ref{ex:building the tree}, we can choose interval generators of $Q_1$, of $Q_2$, and of $Q_3$, but we have no choice when it comes to interval generators of $Q_4$. The $8$ interval generators of $P$ are constructed from the components $Q_i$ as follows.
\begin{align*}
2413 \rightsquigarrow 245613 \rightsquigarrow 2567134 \rightsquigarrow 26578134\\
2413 \rightsquigarrow 245613 \rightsquigarrow 2567143 \rightsquigarrow 26578143\\
2413 \rightsquigarrow 265413 \rightsquigarrow 2765134 \rightsquigarrow 28756134\\
2413 \rightsquigarrow 265413 \rightsquigarrow 2765143 \rightsquigarrow 28756143\\
3142 \rightsquigarrow 314562 \rightsquigarrow 3415672 \rightsquigarrow 34165782\\
3142 \rightsquigarrow 314562 \rightsquigarrow 4315672 \rightsquigarrow 43165782\\
3142 \rightsquigarrow 316542 \rightsquigarrow 3417652 \rightsquigarrow 34187562\\
3142 \rightsquigarrow 316542 \rightsquigarrow 4317652 \rightsquigarrow 43187562
\end{align*}
\end{example}

The definition of $\varepsilon_i$ in Theorem~\ref{thm:number of permutations for a poset} has immediate implications for a certain family of interval posets (to be studied in more depth in the next section).

\begin{corollary}\label{cor:no fruitful means 2 permutations}
An interval poset with no fruitful elements has exactly two interval generators. Moreover, the only interval posets with exactly two interval generators are $\Lambda_4$ and those interval posets with no fruitful elements.
\end{corollary}

\section{Special families of interval posets}\label{sec:special families}

The poset types discussed in Section~\ref{sec:characterizing} motivate us to analyze certain types of interval posets in further detail. Of the many options for what to study, we focus on three:
\begin{itemize}
\item trees,
\item binary posets (that is, those with no fruitful elements), and
\item binary trees.
\end{itemize}

We begin by classifying tree interval posets. Theorem~\ref{thm:interval poset characterization} suggests that argyle posets are what need to be avoided, and Proposition~\ref{prop:argyle posets are interval posets} hints at how that can be done.

\begin{theorem}\label{thm:interval trees}
$\poset(w)$ is a tree interval poset if and only if $w$ contains no interval of the form $p_1\oplus p_2 \oplus p_3$ or $p_1\ominus p_2 \ominus p_3$.
\end{theorem}

\begin{proof}
$\poset(w)$ is a tree if and only if $\poset(w)$ can be constructed without using argyle posets that have at least three minimal elements. By Theorem~\ref{thm:interval poset characterization}, this is possible if and only if $w$ has no interval that is an inflation of $123$ or $321$.
\end{proof}

The previous result can also be stated in terms of the bivincular pattern containment introduced in \cite{bousquet-melou claesson dukes kitaev}.

Binary interval posets can also be characterized by patterns, this time by classical pattern avoidance.

\begin{theorem}\label{thm:binary}
$\poset(w)$ is binary if and only if $w$ avoids the patterns $2413$ and $3142$.
\end{theorem}

\begin{proof}
Any permutation that is an inflation of a simple permutation has a vertex of degree equal to the size of the simple permutation. Thus, an interval poset is binary if and only if every node represents either a direct sum or a skew sum. Such permutations are exactly the \emph{separable permutations}, which are characterized by avoidance of $2413$ and $3142$ (see \cite[P0013]{dppa}).
\end{proof}

Separable permutations were defined in \cite{bose buss lubiw} using the language of ``separating trees.'' These trees bear some resemblance to the interval posets we define here, but they are not the same. For example, the interval poset of the the separable permutation $1234$ is not a tree. Separable permutations are enumerated by the large Schroeder numbers \cite[A006318]{oeis}. Corollary~\ref{cor:no fruitful means 2 permutations} allows us to count binary interval posets (whether or not they are trees).

\begin{corollary}\label{cor:binary enumeration}
Fix $n \ge 2$. The number of binary interval posets with $n$ minimal elements is $(S_n)/2$, where $S_n$ is the $n$th large Schroeder number.
\end{corollary}

The final special family that we study can be classified by a simple observation.

\begin{corollary}\label{cor:binary trees}
$\poset(w)$ is a binary tree interval poset if and only if $w$ avoids $2413$, $3142$, and has no interval that can be written as the direct or skew sum of three or more elements (equivalently, avoids $2413$, $3142$, and any inflation of the bivincular patterns $\ub{\ob{123}}$ and $\ub{\ob{321}}$).
\end{corollary}

\begin{proof}
An interval poset is a binary tree interval poset if and only if it is both a tree and a binary interval poset.
\end{proof}

A binary tree interval poset can be constructed in a single step. This allows us to enumerate the permutations described in Corollary~\ref{cor:binary trees}.

\begin{corollary}\label{cor:binary tree enumeration}
Fix $n\ge 2$.
$$
\left|\begin{Bmatrix}\text{binary tree interval posets}\\ \text{with $n$ minimal elements}\end{Bmatrix}\right| =
\left| \left\{w \in \mf{S}_n : \begin{matrix}w \text{ avoids } 2413, 3142 \text{ and}\\ \text{ inflations of } \ub{\ob{123}} \text{ and } \ub{\ob{321}}\end{matrix}\right\} \right| = 2C_{n-1}$$
where $C_n$ is the $n$th Catalan number.
\end{corollary}

\begin{proof}
There are $C_{n-1}$ full binary trees with $n$ leaves (see, for example, \cite{stanley catalan}).
By Proposition~\ref{prop:binary trees are interval posets} and Corollary~\ref{cor:no fruitful means 2 permutations}, each binary tree interval poset is generated by exactly $2$ permutations. The result then follows from Corollary~\ref{cor:binary trees}.
\end{proof}

\section{Directions for further research}\label{sec:further}

There are many directions in which to continue the study of interval posets for permutations. We highlight a selection of such questions here.

In Section~\ref{sec:special families}, binary interval posets and binary tree interval posets were enumerated. An enumeration of tree interval posets, remains elusive.

\begin{question}
How many tree interval posets have $n$ minimal elements?
\end{question}

Corollary~\ref{cor:no fruitful means 2 permutations} describes the interval posets with exactly two interval generators (necessarily $w$ and $w^{R}$ for some $w$).

\begin{question}
How many interval posets have exactly two interval generators?
\end{question}

When a poset has exactly two interval generators, we know how those two generators are related: they are reverses of each other. Beyond this restricted setting, though, we wonder what other relations might occur.

\begin{question}
What properties are shared among the interval generators of a poset $P$? That is, what can be said about the set $\interval(P)$, besides that it is closed under reversal?
\end{question}

\section*{Acknowledgements}

I am very grateful to the thoughtful comments of anonymous referees, particularly during this challenging time.

\end{document}